\documentclass[12pt]{amsart}
\usepackage[a4paper,margin=2.5cm]{geometry}
\usepackage{graphicx,color}
\usepackage{amssymb}
\usepackage{epstopdf}
\usepackage{amsthm}
\usepackage{marginnote}
\usepackage{enumerate}
\usepackage{tikz}

\usepackage{soul,xcolor}
\setstcolor{red}

\DeclareMathOperator{\rk}{rk}

\newcommand{\contract}{/} 
\newcommand{\delete}{\setminus} 
\newcommand{\cut}[1]{\mathrm{cutdom}(#1)} 
\newcommand{\conv}{\mathrm{conv}} 
\newcommand{\inp}[2]{\langle #1, #2\rangle} 
\newcommand{\level}{\mathrm{lev}} 
\newcommand{\subtour}[1]{\mathrm{subtour}(#1)} 
\newcommand{\GTSP}[1]{\mathrm{GTSP}(#1)} 

\newcommand{\Mone}{$\Theta$} 

\newcommand{\R}{\mathbb{R}} 
\newcommand{\Z}{\mathbb{Z}} 

\newtheorem{thm}{Theorem}
\newtheorem{prop}[thm]{Proposition}
\newtheorem{lem}[thm]{Lemma}

\newtheorem{rem}[thm]{Remark}
\newtheorem*{claim*}{Claim}

\usetikzlibrary{decorations.pathmorphing}
\usetikzlibrary{calc}
\tikzset{snake it/.style={decorate, decoration=snake}}

\title{Cut dominants and forbidden minors}
\author{Michele Conforti \and Samuel Fiorini \and Kanstantsin Pashkovich}

\begin{document}
\maketitle

\begin{abstract} The cut dominant of a graph is the unbounded polyhedron whose points are all those that dominate some convex combination of proper cuts. Minimizing a nonnegative linear function over the cut dominant is equivalent to finding a minimum weight cut in the graph. We give a forbidden-minor characterization of the graphs whose cut dominant can be defined by inequalities with integer coefficients and right-hand side at most $2$. Our result is related to the forbidden-minor characterization of TSP-perfect graphs by Fonlupt and Naddef (Math.\ Prog., 1992). We prove that our result implies theirs, with a shorter proof. Furthermore, we establish general properties of forbidden minors for right-hand sides larger than $2$.
\end{abstract}

\section{Introduction}

Let $G=(V,E)$ be an undirected graph. Parallel edges and loops are allowed in $G$ but we assume $|V| \geqslant 2$. Sometimes we use the notation $V(G)$ and $E(G)$ to denote the node and edge set of $G$, respectively, and $|G|$ and $||G||$ to denote the size of these sets.

For $S \subseteq V$, $\delta(S)$ is the set of edges with one endnode in $S$ and one endnode in $\overline{S} := V \smallsetminus S$. When $S = \{v\}$ is a singleton, we write $\delta(v)$ to mean $\delta(\{v\})$. An edge set $E'\subseteq E$ is a \emph{cut} if  $E'=\delta(S)$ for some $S\subseteq V$ and a \emph{proper} cut if one can take 
$\varnothing \neq S \subsetneq V$. Notice that $\varnothing$ is a proper cut of $G$ if and only if $G$ is a disconnected graph.

\subsection{Cut dominant} \label{sec:intro-cut-dominant}

The \emph{characteristic vector} $\chi^{E'} \in \{0,1\}^{E}$ of an edge set $E' \subseteq E$ is defined as
\begin{equation*}
\chi^{E'}(e):=\begin{cases}
1&\text{if } e \in E'\\
0&\text{otherwise}\,.
\end{cases}
\end{equation*}

The \emph{cut dominant} of $G=(V,E)$ is the polyhedron defined as
$$
\cut{G} := \conv \{\chi^{\delta(S)} \in \{0,1\}^{E} \mid \varnothing \ne S \subsetneq V\} + \R^{E}_+\,.
$$
In other words, $\cut{G}$ is a dominant polyhedron (see \cite[Section 5.8]{SchrijverBook} for properties of dominant polyhedra), whose vertices are the characteristic vectors of inclusionwise minimal proper cuts of $G$ and whose extreme rays are the unit vectors $\chi^{\{e\}}$, $e\in E$. Notice that  $\cut{G}=\R^{E}_+$ if and only if $G$ is a disconnected graph.

Given $c\in \R^E$, we let
\begin{equation*}
\lambda^c(G) := \min \{\inp{c}{x} \mid x \in\cut{G}\}\,.
\end{equation*}
Since $\cut{G}$ is a dominant polyhedron, $\lambda^c(G)$ is finite if and only if $c$ is nonnegative and in this case $\lambda^c(G)$ is the minimum cost of a proper cut in $G$. We call such a cut \emph{minimum}, or sometimes \emph{minimum with respect to $c$} to avoid confusion.

Although $\lambda^c(G)$ can be computed efficiently, a system of linear inequalities that describes $\cut{G}$ in its original space $\R^E$ is unknown. This is our first motivation to study the facets of the cut dominant. Our second motivation, explained below, is the connection to the traveling salesman problem: via blocking polarity, the facets of the cut dominant correspond exactly to vertices of the subtour elimination polyhedron. 

If one allows extra variables, one can find a complete linear description of $\cut{G}$. One such extended formulation of $\cut{G}$ is obtained by picking an arbitrary root $r \in V$ and bidirecting the edges of $G = (V,E)$, which gives a digraph $D = (V,A)$. The extended formulation has two variables $y(u,v)$ and $y(v,u)$ for each edge $uv$ of $G$, and one constraint per $r$-arborescence\footnote{A subset $B \subseteq A$ is \emph{$r$-arborescence} if it is obtained from a spanning tree of $G$ by directing all edges away from the root $r$.} of $D$, besides nonnegativity and the equations linking the $x$-variables to the $y$-variables:
\begin{align}
\label{eq:EF_outbranching}
\sum_{a \in B} y(a) &\geqslant 1 \quad \text{ for all $r$-arborescences } B\\
\label{eq:EF_nonneg}
y(a) &\geqslant 0 \quad \text{ for all $a \in A$}\\[1ex]
\label{eq:EF_proj}
x(uv) &= y(u,v) + y(v,u) \quad \text{ for all $uv \in E$}\,.
\end{align}
By a result of Edmonds~\cite{Edmonds73}, the vertices of the dominant polyhedron defined by \eqref{eq:EF_outbranching}--\eqref{eq:EF_nonneg} are the incidence vectors of the nonempty cuts of $G$, directed away from $r$. This implies that this polyhedron projects to $\cut{G}$ via \eqref{eq:EF_proj}.

We resume our discussion of $\cut{G}$ in its original space. It is easy to see that a linear inequality $\inp{c}{x} \geqslant \lambda$ is valid for $\cut{G}$ if and only if $c \in \R^E_+$ and $\lambda \leqslant \lambda^c(G)$.


\begin{lem}[\cite{CFN85}]\label{lem:facet}
Let $\inp{c}{x} \geqslant \lambda$ be a valid inequality for $\cut{G}$ with $\lambda > 0$, and let $E^c := \{e\in E \mid c(e) \neq 0\}$ be the \emph{support} of $c$. Then $\inp{c}{x} \geqslant \lambda$ is facet-defining if and only if $\lambda = \lambda^c(G)$ and there exists a family $\mathcal F$ of $|E^c|$ subsets of $V$ such that:
\begin{enumerate}[(i)]
\item For each $S\in \mathcal F$, the cut $\delta(S)$ is minimum with respect to\ $c$.
\item The vectors $\chi^{\delta(S) \cap E^c} $, $ S \in \mathcal{F}$ are linearly independent.
\end{enumerate}
Furthermore the family $\mathcal F$ can be chosen to be laminar.\footnote{Family $\mathcal F$ is \emph{laminar} if every two sets in $\mathcal F$ are either disjoint or comparable for inclusion.}
\end{lem}

Since $\cut{G}$ is full-dimensional, facet-defining inequalities are uniquely defined up to scaling by a positive number. In particular, every facet-defining inequality $\inp{c}{x} \geqslant \lambda$ can be written in a unique way so that its cofficients are integral and relatively prime. Such an inequality is said to be in \emph{minimum integer form}. 


For a nonnegative integer $k$, graph $G$ is called a \emph{$k$-graph} if every facet $\inp{c}{x} \geqslant \lambda$ of $\cut{G}$ satisfies $\lambda \leqslant k$, when written in minimum integer form. We let $k^*(G)$ denote the minimum $k$ such that $G$ is a $k$-graph. 

It follows from Lemma~\ref{lem:facet} that $k^*(G)=0$ if $G$ is disconnected. Moreover, $k^*(G)=1$ if $G$ is a tree. This is due to the fact that if $G$ is a tree and $\inp{c}{x} \geqslant \lambda$ is a facet-defining inequality of $\cut{G}$ in minimum integer form with $\lambda > 0$, then a cut is minimum only if it is a fundamental cut of $G$. Hence, by Lemma~\ref{lem:facet}, the minimum cuts are exactly the fundamental cuts. Therefore, $c(e) = \lambda = 1$ for all $e \in E$. Cornu\'ejols, Fonlupt and Naddef~\cite{CFN85} show that $k^*(G) \leqslant 2$ whenever $G$ is series-parallel.

The next result of Conforti, Rinaldi and Wolsey~\cite{CRW04} shows that $k^*(G)=1$ or $k^*(G)$ is even.

\begin{thm}[\cite{CRW04}]
 If $\inp{c}{x} \geqslant \lambda$ is a facet-defining inequality for $\cut{G}$ in minimum integer form and $\lambda$ is odd, then $\lambda = 1$ and $E^c$ is a spanning tree of $G$.
\end{thm}

From Remarks~\ref{rem:loops} and~\ref{rem:skeleton} it follows that if the simple graph $G$ is obtained from another graph $G'$ by deleting all loops and keeping one edge in each set of parallel edges, then $k^*(G) = k^*(G')$. In particular, $k^*(G')=k^*(G)=1$ if $G'$ is a tree with some edges duplicated and loops added.

A \emph{minor} of $G$ is a graph that can be obtained from $G$ by any sequence of edge contractions, edge deletions and node deletions. For $e \in E$, let $G \delete e$ and $G \contract e$ be the graphs obtained from $G$ by deleting and contracting respectively the edge $e$, and for $v\in V$ let $G-v$ be the graph obtained from $G$ by deleting $v$.

\begin{lem}\label{lem:deletion_contraction}
If $G'$ is obtained from $G$ by edge contractions and edge deletions, then $k^*(G') \leqslant k^*(G)$.
\end{lem}
\begin{proof}
It suffices to verify the statement for $G'=G \contract e$ and for $G'=G \delete e$, where $e\in E(G)$.

For $G'= G \contract e$, the polyhedron $\cut{G'}$ is the intersection of the face of $\cut{G}$ defined by $x(e) \geqslant 0$. Thus, a linear description of $\cut{G'}$ can be obtained from a linear description of $\cut{G}$ by deleting the term $c(e) x(e)$ in each inequality $\inp{c}{x} \geqslant \lambda$ in the description.

For $G' := G \delete e$, the polyhedron $\cut{G'}$ is the projection of $\cut{G}$ onto $\R^{E(G')}$. In this case, a linear description of $\cut{G'}$ can be obtained from a linear description of $\cut{G}$ by eliminating the variable $x(e)$ from the system. Since all inequalities $\inp{c}{x}\geqslant \lambda$ in a linear description of $\cut{G}$ have nonnegative coefficients, the elimination of $x(e)$ simply corresponds to removing from the system all inequalities with $c(e) > 0$. 

In both cases, $k^*(G') \leqslant k^*(G)$.
\end{proof}

\begin{lem}\label{lem:minordom}
If $G'$ is a minor of $G$, and $G$ is connected, then $k^*(G') \leqslant k^*(G)$.
\end{lem}
\begin{proof}
We may assume that $G'$ is connected because otherwise $k^*(G') = 0$ and the statement holds trivially. Now, since $G$ and $G'$ are connected, $G'$ can be obtained from $G$ by edge contractions and edge deletions only (no node deletions) and the statement follows from Lemma~\ref{lem:deletion_contraction}
\end{proof}

It follows from Lemma~\ref{lem:minordom} that $k^*(G) = 0$ only if $G$ is disconnected, since $k^*(K_2) = 1$. Moreover, $k^*(G) = 1$ only if $G$ is a graph obtained from a tree by duplicating some edges and adding some loops, since $k^*(K_3) = 2$. (As usual, $K_n$ denotes the complete graph on $n$ nodes.) 

Given a nonnegative integer $k$, $G$ is a \emph{minor-minimal non-$k$-graph} if $k^*(G) > k$ but $k^*(G') \leqslant k$ for every proper minor $G'$ of $G$. It follows from Lemma~\ref{lem:minordom} and the graph minor theorem of Robertson and Seymour~\cite{RS04} that for fixed $k$, the list of minor-minimal non-$k$-graphs is finite. Remark that the only minor-minimal non-$0$-graph is $K_2$ and the only minor-minimal non-$1$-graph is $K_3$. The main result of this paper (Theorem \ref{thm:2-graphs} below) characterizes the minor-minimal non-$2$-graphs.



Consider now the following graphs: The \emph{prism} is the triangular prism, i.e., the complement of the $6$-cycle $C_6$; and the \emph{pyramid} is the graph obtained from $K_4$ by choosing a node $v\in V(K_4)$ and subdividing each of the three edges incident to $v$ (see Figure~\ref{fig:prism_pyramid}).

\begin{figure}[ht]
\centering
\begin{tikzpicture}[inner sep = 2.5pt, scale=1.5]
\def\rSmall{.17}
\def\rEllipsSmall{.35}
\def\rEllipsBig{1.4}
\tikzstyle{vtx}=[circle,draw,thick,fill=gray!25]
\tikzstyle{cut}=[blue]
\node[vtx] (a) at (0,0) {};
\node[vtx] (b) at (1.95,-.75) {};
\node[vtx] (c) at (1.15,.25) {};
\node[vtx] (a') at (0,2) {};
\node[vtx] (b') at (1.95,1.25) {};
\node[vtx] (c') at (1.15,2.25) {};
\draw[thick] (a) -- (b) -- (c) -- (a);
\draw[thick] (a') -- (b') -- (c') -- (a');
\draw[very thick, magenta] (a) -- (a');
\draw[very thick, magenta] (b) -- (b');
\draw[very thick, magenta] (c) -- (c');
\draw[cut] (a) circle (\rSmall);
\draw[cut] (b) circle (\rSmall);
\draw[cut] (c) circle (\rSmall);
\draw[cut] (a') circle (\rSmall);
\draw[cut] (b') circle (\rSmall);
\draw[cut] (c') circle (\rSmall);
\draw[cut] ($(a)!0.5!(a')$) ellipse ({\rEllipsSmall} and {\rEllipsBig});
\draw[cut] ($(b)!0.5!(b')$) ellipse ({\rEllipsSmall} and {\rEllipsBig});
\draw[cut] ($(c)!0.5!(c')$) ellipse ({\rEllipsSmall} and {\rEllipsBig});
\def\rEllipsBig{1.1}
\def\rEllipsSmall{.32}
\node[vtx] (e) at (6,2.25) {};
\node[vtx] (f) at (6,.825) {};
\node[vtx] (g) at (5.5,1.25) {};
\node[vtx] (h) at (6.5,1.25) {};
\node[vtx] (i) at (5,0) {};
\node[vtx] (j) at (7,0) {};
\node[vtx] (k) at (6,-.75) {};
\draw[thick] (k) -- (i) -- (j)--(k);
\draw[very thick, magenta] (e) -- (f) -- (k);
\draw[very thick, magenta] (e) -- (h) -- (j);
\draw[very thick, magenta] (e) -- (g) -- (i);
\draw[cut] (f) circle (\rSmall);
\draw[cut] (g) circle (\rSmall);
\draw[cut] (h) circle (\rSmall);
\draw[cut] (i) circle (\rSmall);
\draw[cut] (j) circle (\rSmall);
\draw[cut] (k) circle (\rSmall);
\draw[cut] ($(f)!0.5!(k)$) ellipse ({\rEllipsSmall} and {\rEllipsBig});
\draw[cut, rotate around={22:($(h)!0.5!(j)$)}] ($(h)!0.5!(j)$) ellipse ({\rEllipsSmall} and {\rEllipsBig});
\draw[cut, rotate around={-22:($(g)!0.5!(i)$)}] ($(g)!0.5!(i)$) ellipse ({\rEllipsSmall} and {\rEllipsBig});

\end{tikzpicture}
\caption{The prism (left) and pyramid (right), together with two laminar families of cuts (blue) proving that each of the graphs is a non-$2$-graph.}
\label{fig:prism_pyramid}
\end{figure}
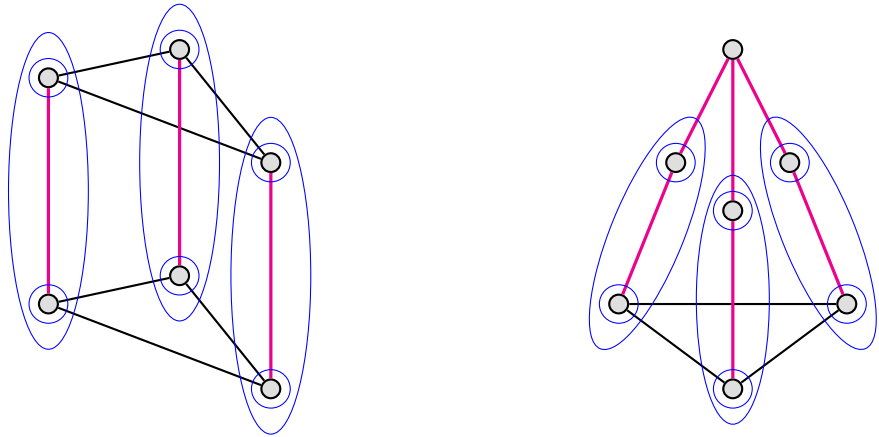

We claim that both the prism and the pyramid are non-$2$-graphs. In order to prove this, we exhibit a facet-defining inequality for $\cut{G}$ in minimum integer form and with a right-hand side of $4$, where $G$ is either the prism or the pyramid. Let $c \in \R^E$ be such that $c(e) := 1$ if edge $e$ belongs to a triangle and $c(e):=2$ otherwise, see Figure~\ref{fig:prism_pyramid} (cost-$1$ edges are those colored black and cost-$2$ edges magenta), and let $\lambda := \lambda^c(G) = 4$. Then $\inp{c}{x} \geqslant \lambda$ is obviously in minimum integer form. A laminar family $\mathcal F$ consisting of $|E| = 9$ subsets $S \subseteq V$ such that the cuts $\delta(S)$, $S \in \mathcal F$ are linearly independent\footnote{ We do not distinguish cuts from their characteristic vectors. We say that the cuts $\delta(S)$, $S \in \mathcal{F}$ are linearly independent if the vectors $\chi^{\delta(S)}$, $S \in \mathcal{F}$ are linearly independent.} minimum (proper) cuts is represented in Figure \ref{fig:prism_pyramid} (sets $S \in \mathcal{F}$ are those colored blue). Therefore, by Lemma~\ref{lem:facet}, $\inp{c}{x} \geqslant \lambda$ is a facet-defining inequality for $\cut{G}$. We conclude that $G$ is a non-$2$-graph.

Moreover, it can be checked that  $k^*(G \contract e) =2$ and $k^*(G \delete e)= 2$ for every edge $e$, where $G$ is either the prism or pyramid. Thus the prism and  pyramid  are both \emph{minor-minimal} non-$2$-graphs. (The minor-minimality of the prism and pyramid is not used in our proof of Theorem~\ref{thm:2-graphs} and actually follows from it.)


\begin{thm}[Main Theorem] \label{thm:2-graphs}
The minor-minimal non-$2$-graphs are the prism and the pyramid.
\end{thm}

\subsection{Subtour elimination polyhedron} \label{sec:intro-subtour}

A \emph{tour} of $G$ is a closed walk that visits every node at least once. Every tour $T$ can be encoded by means of an integer vector $\chi^T \in \Z^E_+$ by letting $\chi^T(e)$ count the number of times edge $e$ is traversed by tour $T$. The resulting set of vectors $X(G)$ is the set of vectors $x \in \Z^E_+$ with connected support, such that $x(\delta(v))$ is even for every node $v$. (As usual, $x(E'):=\sum_{e\in E'} x(e)$ for $E'\subseteq E$.)

The \emph{graphical traveling salesman polyhedron} is the polyhedron defined as
$$
\GTSP{G} := \conv \{\chi^T \in \Z^E_+ \mid T \text{ is a tour of } G\} = \conv (X(G))\,.
$$
This polyhedron is of dominant type because if $x \in X(G)$ then $x + 2 \chi^{\{e\}} \in X(G)$. 

The \emph{subtour elimination relaxation} of the graphical traveling salesman polyhedron is the polyhedron defined as
$$
\subtour{G} := \{x \in \R^E_+ \mid x(\delta(S))\geqslant 2,\;\varnothing\ne S\subsetneq V\}\,.
$$
Clearly $\GTSP{G} \subseteq \subtour{G}$. Fonlupt and Naddef \cite{FN92} prove the following:


\begin{thm} \label{thm:Fon-Nadd}
We have $\GTSP{G}=\subtour{G}$ if and only if $G$ does not have a prism or pyramid or \Mone{} (see Figure~\ref{fig:M_1}) minor.
\end{thm}

\begin{figure}[ht]
\centering
\begin{tikzpicture}[inner sep = 2.5pt,thick, scale=1.5,rotate=90]
\tikzstyle{vtx}=[circle,draw,thick,fill=gray!25]
\node[vtx] (s) at (0,0) {};
\node[vtx] (t) at (0,3) {};
\node[vtx] (v1) at (-1,1) {};
\node[vtx] (v2) at (0,1) {};
\node[vtx] (v3) at (1,1) {};
\node[vtx] (w1) at (-1,2) {};
\node[vtx] (w2) at (0,2) {};
\node[vtx] (w3) at (1,2) {};
\draw (s) -- (v1) -- (w1) -- (t); 
\draw (s) -- (v2) -- (w2) -- (t); 
\draw (s) -- (v3) -- (w3) -- (t); 
\end{tikzpicture}
\caption{The \Mone{} graph, third forbidden minor for $\GTSP{P} = \subtour{P}$.}
\label{fig:M_1}
\end{figure}
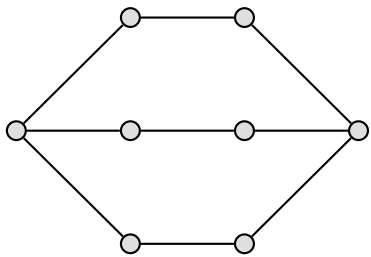

By blocking polarity \cite[Section 5.8]{SchrijverBook}, the vertices of the subtour elimination relaxation yield an irredundant description of the cut dominant:
$$
\cut{G} = \{x \in \R^E_+ \mid \inp{c}{x}\geqslant 2,\;c\mbox{ is a vertex of }\subtour{G}\}.
$$
(In the above description, facet-defining inequalities with positive right hand side $\lambda$ are normalized so that $\lambda = 2$.) Therefore the result of Fonlupt and Naddef shows that if $G$ does not contain one of the minors listed in Theorem~\ref{thm:Fon-Nadd}, then $G$ is a $2$-graph. However, as we prove, this still holds when $G$ does not have a prism or pyramid minor. In terms of the blocker,  this means that the subtour elimination relaxation is guaranteed to be integral, even when $G$ has a \Mone{} minor, but might have vertices that do not correspond to tours.

\subsection{Other related work} \label{sec:intro-other-related-work}

Dirac \cite{Dirac63} and Lov\'asz \cite{Lovasz65} have characterized the graphs that do not contain two node-disjoint cycles. Fumei Lam in her doctoral dissertation~\cite{Lam05} gives a proof of Theorem \ref{thm:Fon-Nadd} that uses this characterization. We can also prove Theorem \ref{thm:2-graphs} using the Dirac-Lov\'asz characterization. However, the proof involves an extensive case analysis (albeit a lot shorter than the one in \cite{Lam05}) and is less elegant than the proof presented here.

\subsection{Outline} \label{sec:intro-outline}

We start Section~\ref{sec:general} by giving basic properties of facet-defining inequalities of the cut dominant of $G$. In particular we give three remarks for dealing with loops, parallel edges and cutnodes in $G$. 

For the rest of the section and for the following one, we consider a minor-minimal non-$k$-graph $G$. Since $G$ is a non-$k$-graph, $\cut{G}$ has a \emph{witness}, that is, a facet-defining inequality $\inp{c}{x} \geqslant k$ for $\cut{G}$ such that its right hand side exceeds $k$ when it is scaled to its minimum integer form.\footnote{Although this definition depends on $k$, the value of $k$ will always be clear from the context.} Corresponding to this witness, $G$ has a laminar family $\mathcal{F}$ such that $\{\delta(S) \mid S \in \mathcal{F}\}$ is a basis of minimum cuts with respect to $c$. 

We give properties of $G$, witness $\inp{c}{x} \geqslant k$ and laminar family $\mathcal{F}$. Specifically, these properties include a characterization of the local structure of level-$0$ and level-$1$ sets in $\mathcal{F}$, where level-$0$ sets are the inclusionwise minimal sets in $\mathcal{F}$  and, for $i \in \Z_+$, level-$(i+1)$ sets are those for which $i$ is the maximum level of a properly contained set. 

We prove that level-$0$ sets are singletons and level-$1$ sets are pairs of nodes linked by an edge of cost $\frac{k}{2}$.


In Section~\ref{sec:non-2-graphs}, we continue the analysis under the further assumption that $k = 2$. First, we prove that $c$ is half-integral, so that $G$ is a $4$-graph. Next, we prove that $G$ is essentially $3$-connected in the sense that if $\{u,v\}$ is a $2$-cutset, then $G - \{u,v\}$ has exactly two components and one of them is formed by a unique node that is adjacent to both $u$ and $v$, while $u$ and $v$ are not adjacent. Finally, we prove that $\mathcal{F}$ can be assumed to have no level-$2$ set. The last step in the proof of Theorem~\ref{thm:2-graphs} combines all these properties in a global argument to prove that $G$ has a prism or pyramid minor.

In Section~\ref{sec:equivalence}, we show that our main result (Theorem~\ref{thm:2-graphs}) quite directly implies the main result of Fonlupt and Naddef~\cite{FN92} (Theorem~\ref{thm:Fon-Nadd}). We prove also that, using our $2$-cutset lemma (Lemma~\ref{lem:2-cutset_bis}) as well as half-integrality of witnesses (Lemma~\ref{lem:half_integral}), the main result of Fonlupt and Naddef implies our main result.

Finally, concluding remarks are given in Section~\ref{sec:concluding_remarks}.


\section{General properties} \label{sec:general}

\subsection{Properties of facet-defining inequalities} \label{sec:general-FDIs}

Consider a graph $G = (V,E)$. The first two remarks imply that if $G$ is the simple graph obtained from graph $G'$ by keeping exactly one edge from each set of parallel edges and deleting all loops, then $G$ satisfies $k^*(G) = k^*(G')$.
\begin{rem} \label{rem:loops}
Let $G$ be a graph, let $G'$ be the graph obtained from $G$ by adding a loop $f$. Suppose that
$$
\sum_{e \in E(G)} c^{i}(e) x(e) \geqslant \lambda_i,\,i\in I;\;\;x(e) \geqslant 0,\;e \in E(G)
$$
is an irredundant linear description of $\cut{G}$. Then the system 
$$
\sum_{e \in E(G')\atop e \neq f} c^{i}(e) x(e) \geqslant \lambda_i,\,i\in I;\;\;x(e) \geqslant 0,\;e \in E(G')
$$
is an irredundant linear description of $\cut{G'}$.
\end{rem}

\begin{rem} \label{rem:skeleton}
Let $G$ be a graph, let $f$ be an edge of $G$ and let $G'$ be the graph obtained from $G$ by adding an edge $f'$ parallel to $f$. Suppose that
$$
\sum_{e \in E(G)} c^{i}(e) x(e) \geqslant \lambda_i,\,i\in I;\;\;x(e) \geqslant 0,\;e \in E(G)
$$
is an irredundant linear description of $\cut{G}$. Then the system 
\begin{align*}
&c^{i}(f)x(f) + \sum_{e \in E(G') \atop e \neq f,f'} c^{i}(e)x(e) \geqslant \lambda_i,\,i\in I, c^i(f) > 0;\\
&c^{i}(f)x(f') + \sum_{e \in E(G') \atop e \neq f,f'} c^{i}(e)x(e) \geqslant \lambda_i,\,i\in I, c^i(f) > 0;\\
&\sum_{e \in E(G') \atop e \neq f,f'} c^{i}(e)x(e) \geqslant \lambda_i,\,i\in I, c^i(f) = 0;\\
&x(e) \geqslant 0,\;e \in E(G')
\end{align*}
is an irredundant linear description of $\cut{G'}$.
\end{rem} 

The next remark explains how the facets of the cut dominant of $G$ can be obtained from the facets of the cut dominant of each of its blocks, in case $G$ is not $2$-connected.


\begin{rem}\label{rem:cutdom-cutnode} 
Let $G$ be a graph that can be obtained from two disjoint graphs $G_1$, $G_2$ (with at least two nodes each), by selecting a node in $G_1$ and a node in $G_2$ and identifying them into a node $v$, so that $v$ is a cutnode of $G$. Moreover, let 
\begin{align*}
&\inp{c^{1,i}}{x^1} \geqslant k,\,i\in I;\;\;x^1(e) \geqslant 0,\;e \in E(G_1)\\
&\inp{c^{2,j}}{x^2} \geqslant k,\,j\in J;\;\;x^2(e) \geqslant 0,\;e \in E(G_2)
\end{align*}
be irredundant systems of inequalities describing $\cut{G_1}$ and $\cut{G_2}$ respectively, where $k > 0$ is arbitrary. The following system of inequalities provides an irredundant description of $\cut{G}$:
$$
\inp{c^{1,i}}{x^1} + \inp{c^{2,j}}{x^2} \geqslant k,\,i\in I,\,j \in J;\;\;x^1(e) \geqslant 0,\;e \in E(G_1);\;\;x^2(e) \geqslant 0,\;e \in E(G_2)\,.
$$
\end{rem}

When $G$ is obtained from disjoint graphs $G_1$ and $G_2$ by identifying two or more nodes, it appears to be significantly harder to derive an inequality description of $\cut{G}$ from the descriptions of $\cut{G_1}$ and $\cut{G_2}$.

Next, let $\inp{c}{x}\geqslant \lambda$ be a valid inequality for $\cut{G}$ with $\lambda = \lambda^c(G) > 0$. Let $G^c$ denote the graph obtained from $G$ by deleting all edges with $c(e) = 0$. Thus $G^c = (V,E^c)$. Lemma~\ref{lem:facet} shows that $\inp{c}{x} \geqslant \lambda$ is facet-defining for $\cut{G}$ if and only if it is facet-defining for $\cut{G^c}$. 


The following lemma states basic properties that are used later, often without explicit reference.

\begin{lem}\label{lem:facet-basic} Let $G$ be a graph, $\inp{c}{x} \geqslant \lambda$ be a facet-defining inequality for $\cut{G}$ with $\lambda > 0$ and $E^c=E$, and let $\{\delta(S) \mid S \in \mathcal{F}\}$ be a basis of minimum cuts (with respect to $c$).

\begin{enumerate}[(i)]
\item For every $S \in \mathcal{F}$, the induced subgraphs $G[S]$ and $G[\overline{S}]$ are both connected.
\item Graph $G$ is simple, that is, $G$ contains no pair of parallel edges and no loops.
\item For every $e \in E$ there is at least one $S \in \mathcal F$ such that $e \in \delta(S)$.
\end{enumerate}
\end{lem}
\begin{proof}
(i) follows from the fact that $\delta(S)$ is a minimum cut. (ii) follows from Remarks~\ref{rem:loops} and~\ref{rem:skeleton}. (iii) follows from the fact that cuts $\delta(S),\;S\in \mathcal F$ form a basis of $\R^E$. Indeed, if (iii) is not satisfied, then every cut $\delta(S)$, $S \in \mathcal{F}$ satisfies $x(e) = 0$, a contradiction.
\end{proof}

For an edge set $E' \subseteq E$, we let $c \contract E'$ denote the restriction of $c$ to $E(G \contract E')$. If $E' = \{e\}$ is a single edge, we also use the simpler notation $c \contract e$ to denote $c \contract \{e\}$. 

\begin{lem} \label{lem:contract}
Let $G$ be a graph, $\inp{c}{x} \geqslant \lambda$ be a facet-defining inequality for $\cut{G}$ with $\lambda > 0$ and $E^c=E$, and let $\{\delta(S) \mid S \in \mathcal{F}\}$ be a basis of minimum cuts. For every $E' \subseteq E$ such that $G \contract E'$ has at least two nodes, the inequality $\inp{c \contract E'}{x} \geqslant \lambda$ is valid for $\cut{G \contract E'}$. Moreover, the dimension of the face of $\cut{G \contract E'}$ this inequality defines is at least
\begin{equation*}
|\{S\in\mathcal{F} \mid \delta(S)\cap E'=\varnothing\}|-1\,.
\end{equation*}
\end{lem}
\begin{proof}
By hypothesis, $\lambda = \lambda^c(G)$. Since every cut in $G \contract E'$ yields a cut of the same cost in $G$, we have $\lambda^c(G) \leqslant \lambda^{c\contract E'}(G\contract E')$. Hence the inequality $\inp{c \contract E'}{x} \geqslant \lambda$ is valid for $\cut{G \contract E'}$. 

Recall that $\{\delta(S) \mid S \in \mathcal{F}\}$ is a basis of minimum cuts. Since  $\{\delta(S)\mid S \in \mathcal{F},\,\delta(S) \cap E'=\varnothing\}$ is a set of linearly independent cuts in $G \contract E'$ which satisfy the inequality $\inp{c \contract E'}{x} \geqslant \lambda$ with equality, the cardinality of such a set is a lower bound on the number of affinely independent points on the face defined by $\inp{c \contract E'}{x} \geqslant \lambda$.
\end{proof}

\medskip
\subsection{Properties of minor-minimal non-$k$-graphs} \label{sec:general-mm_non-k-graphs}

Here we assume $k \geqslant 1$ and consider any minor-mini\-mal non-$k$-graph $G = (V,E)$. Such a graph has a witness, i.e. a facet-defining inequality of the form  $\inp{c}{x} \geqslant k$ such that the right hand side of its minimum integer form is strictly greater than $k$. Below, we establish properties of $G$ and its cut dominant that follow from the minimality of $G$. 
Recall that $c(E') := \sum_{e \in E'} c(e)$ for $E' \subseteq E$ and $c \in \mathbb{R}^{E}$.


\begin{lem} \label{lem:basic_two_cuts}  Let $G$ be a minor-minimal non-$k$-graph, $\inp{c}{x} \geqslant k$ be a witness for $G$ and let $\{\delta(S) \mid S \in \mathcal{F}\}$ be a basis of minimum cuts. Every edge of $G$ belongs to at least two cuts $\delta(S)$, $S \in \mathcal{F}$.
\end{lem}
\begin{proof} 
Because $k\geqslant 1$, the minimality of $G$ implies $|V| \geqslant 3$. Suppose that there exists an edge $e$ that belongs to a unique cut $\delta(S)$, $S \in \mathcal F$. By Lemma~\ref{lem:contract}, $\inp{c \contract e}{x} \geqslant k$ is a facet-defining inequality for $G\contract e$. Let us show that the inequality $\inp{c \contract e}{x} \geqslant k$ is a witness for $G \contract e$. Indeed, if $\lambda > k$ is such that $\inp{\frac{\lambda}{k}c}{x} \geqslant \lambda$ is the minimum integer form of witness $\inp{c}{x} \geqslant k$ for $G$ then $\inp{\frac{\lambda}{k}(c \contract e)}{x} \geqslant \lambda$ is the minimum integer form of inequality $\inp{c \contract e}{x} \geqslant k$. This is due to the fact that $c(e) = k - c(\delta(S) \smallsetminus e)$, thus any common divisor of $\frac{\lambda}{k}(c \contract e)$ and $\lambda$ also divides $\frac{\lambda}{k}c(e)$. Since $G \contract e$ has a witness, it is a non-$k$-graph, contradicting the minimality of $G$.
\end{proof}


\begin{lem} \label{lem:basic_cost_one}
 Let $G$ be a minor-minimal non-$k$-graph and $\inp{c}{x} \geqslant k$ be a witness for $G$. Then $c(e) \leqslant \frac{k}{2}$ for every edge $e$.
\end{lem}
\begin{proof}
Suppose $c(e) > \frac{k}{2}$. Let $\mathcal{F}$ be a laminar family such that $\{\delta(S) \mid S \in \mathcal{F}\}$ is a basis of minimum cuts. By Lemma~\ref{lem:basic_two_cuts}, there are two different sets $S, T \in \mathcal{F}$ such that $e \in \delta(S)$, $\delta(T)$. Because $\mathcal{F}$ is a laminar family, without loss of generality we can assume $S \subsetneq T$. Then
\begin{equation*}
c\big(\delta(T \smallsetminus S)\big) \leqslant c\big(\delta(S)) + c(\delta(T)\big) - 2c(e) = 2k - 2c(e) < k
\end{equation*}
hence the inequality $\inp{c}{x} \geqslant k$ is not valid for $\cut{G}$, a contradiction.
\end{proof}

We associate a \emph{level} in $\Z_+$ for each set of the laminar family $\mathcal{F}$ through the function $\level:\mathcal{F}\rightarrow \Z_+$ (recursively) defined by:
\begin{equation*}
\level(S):=
\begin{cases}
0 & \text{if }S\text{ is an inclusionwise minimal set in }\mathcal{F},\\
\max_{T\subsetneq S, T \in \mathcal{F} } \level(T) +1 &\text{otherwise}\,.
\end{cases}
\end{equation*}

\begin{lem}\label{lem:lev-0} 
Let $G$ be a minor-minimal non-$k$-graph, $\inp{c}{x} \geqslant k$ be a witness for $G$ and $\mathcal{F}$ be a laminar family such that $\{\delta(S) \mid S \in \mathcal{F}\}$ is a basis of minimum cuts. Every level-$0$ set is a singleton.
\end{lem}
\begin{proof}
Consider a level-$0$ set $S\in\mathcal{F}$. Because $\delta(S)$ is a minimum cut, the induced subgraph $G[S]$ is connected. On the other hand, every edge that $G[S]$ might have should belong to at least one cut $\delta(T)$, $T \in \mathcal F$. By laminarity of $\mathcal{F}$, such a cut should satisfy $T \subsetneq S$, thus contradicting the fact that $S$ is level-$0$. We conclude that $|S|=1$.
\end{proof}

Given disjoint subsets $S$, $T$ of $V$, we denote by $\delta(S:T)$ the set of edges with one endpoint in $S$ and the other in $T$. As usual, if $T = \{v\}$ is a singleton, we write $\delta(S:v)$ to mean $\delta(S:\{v\})$.


\begin{lem} \label{lem:level-1}  
Let $G$ be a minor-minimal non-$k$-graph, $\inp{c}{x} \geqslant k$ be a witness for $G$ and $\mathcal{F}$ be a laminar family such that $\{\delta(S) \mid S \in \mathcal{F}\}$ is a basis of minimum cuts. Every level-$1$ set in $\mathcal{F}$ is of the form $\{u,v\}$, where $\{u\}, \{v\} \in \mathcal{F}$ and $uv \in E(G)$. Moreover, $c(uv)=c(\delta(u)\smallsetminus uv)=c(\delta(v)\smallsetminus uv)=\frac{k}{2}$.
\end{lem}
\begin{proof}
Let $S$ be a level-$1$ set in $\mathcal{F}$. Because $\level(S) > 0$, we have $|S| \geqslant 2$.

Because $G[S]$ is connected, $||G[S]|| \geqslant |S|-1$. 

Let $\mathcal{F}_S := \{R\in \mathcal{F} \mid R\subsetneq S\}$. Each $R \in \mathcal{F}_S$ has level $0$ and is thus a singleton, by Lemma~\ref{lem:lev-0}. Hence $|\mathcal {F}_S| \leqslant |S|$. In fact, we have $|\mathcal {F}_S| = |S|$ since otherwise $G[S]$ would have an edge $uv$ with $\{u\} \notin \mathcal{F}_S$. Such an edge would be contained in at most one cut $\delta(T)$, $T \in \mathcal{F}$, in contradiction with Lemma~\ref{lem:basic_two_cuts}. Hence for every node $v \in S$, we have $\{v\} \in \mathcal{F}$.

Because $\{\delta(T) \mid T\in\mathcal{F}\}$ is a basis of minimum cuts of $G$, family $\mathcal{F}$ contains $||G[S]||$ sets $R$ such that $\delta_{G[S]}(R) = \delta(R) \cap E(G[S])$ form a basis of cuts of $G[S]$. Every such set $R$ satisfies $\delta(R)\cap E(G[S]) \neq \varnothing$, which by laminarity of $\mathcal{F}$ implies $R\subsetneq S$, that is, $R \in \mathcal{F}_S$. Thus $|\mathcal{F}_S| \geqslant ||G[S]||$.

Therefore,
\begin{equation} \label{eq:FS-bounds}
|S|-1 \leqslant ||G[S]|| \leqslant |\mathcal{F}_S| = |S|\,.
\end{equation}
Either $||G[S]||=|S|-1$ and $G[S]$ is a tree, or $||G[S]||=|S|$ and $G[S]$ is a connected graph containing exactly one cycle. Because the cuts $\delta_{G[S]}(R)$, $R \in \mathcal{F}_S$ are linearly independent, this cycle is odd.

\begin{claim*} \label{claim:degree-1_in_S}
If $G[S]$ has a node $v$ with a unique neighbor in $G[S]$, say $u$, then 
\begin{equation} \label{eq:degree-1_in_S}
c(vu)=c(\delta(v:\overline{S}))=c(\delta((S\smallsetminus v) : \overline{S}))=\frac{k}{2}\,.
\end{equation}
\end{claim*}
\begin{proof}
Since $S, \{v\} \in \mathcal{F}$ and $\lambda^c(G) = k$ we have
\begin{align*}
&c(v:\overline{S})+c(\delta((S\smallsetminus v) : \overline{S}))=c(\delta(S))=k\\
&c(vu)+c(\delta(v:\overline{S}))=c(\delta(v))=k\quad \text{and}\\
&c(vu)+c(\delta((S\smallsetminus v) : \overline{S}))=c(\delta(S\smallsetminus v))\geqslant k\,.
\end{align*}
By Lemma~\ref{lem:basic_cost_one}, $c(vu) \leqslant \frac{k}{2}$. This implies \eqref{eq:degree-1_in_S}.
\end{proof}


We proceed with the proof of the lemma and distinguish two cases:\medskip

\noindent {\it Case  1:}    \emph{$G[S]$ is a tree.} 
Then $G[S]$ contains at least two leaves, say $v_1$ and $v_t$ where $t := |S|$. Let $v_{2}$, $v_{t-1}$ be their unique neighbors in $G[S]$. By the claim, $c(v_1v_2)=c(v_{t-1}v_t)=\frac{k}{2}$ and moreover $c(\delta(v_1:\overline{S})) = c(\delta(v_t:\overline{S})) = \frac{k}{2}$. Since $c(\delta(S)) = k$, this implies $\delta(S)=\delta(v_1:\overline{S})\cup \delta(v_t:\overline{S})$. This shows that  $G[S]$  contains exactly two leaves. That is, $G[S]$ is a path $P = v_1v_2 \cdots v_{t-1}v_t$. Since $\{v\} \in \mathcal{F}$ for all nodes $v \in S$, we have $c(e) = \frac{k}{2}$ for every $e \in E(P)$.

The lemma holds if $t = 2$. By contradiction, assume $t \geqslant 3$ and notice that 
\begin{equation} \label{eq:P_dependance}
\chi^{\delta(\{v_1,v_2,v_3\})} = \chi^{\delta(v_1)} - \chi^{\delta(v_2)} + \chi^{\delta(v_3)}\,.
\end{equation}
In particular, $t \geqslant 4$ since otherwise $t = 3$, $S = \{v_1,v_2,v_3\}$ and \eqref{eq:P_dependance} contradicts the linear independence of the cuts $\delta(T)$, $T \in \mathcal{F}$. Now, because of \eqref{eq:P_dependance}, we can modify $\mathcal{F}$ by adding $\{v_1,v_2,v_3\}$ and removing $\{v_3\}$ while keeping a laminar family defining a basis of minimum cuts. The new family $\mathcal{F}$ violates Lemma~\ref{lem:basic_two_cuts} because the edge $v_2v_3$ is contained in only one of the cuts it defines, a contradiction.\medskip

\noindent   {\it Case  2:}  \emph{$G[S]$ contains a unique cycle, which is odd.}
We first show that $G[S]$ cannot contain a degree-$1$ node. Assume not and let $v$ be a node of $G[S]$ having $u$ as unique neighbor in $G[S]$. Since $c(vu)=\frac{k}{2}$ by the claim and  $c(\delta(v))=c(\delta(u))=k$, we get that $c(\delta(\{v,u\})) = k$, thus $\{v,u\}$ also defines a minimum cut in $G$. 

We decompose the vector space $\R^E$ into two subspaces $W_1$ and $W_2$ where $W_1 := \{x \in \R^E \mid x(e) = 0,\;e \in E(G[S])\}$ and $W_2 := W_1^\perp = \{x \in \R^E \mid x(e) = 0,\;e \in E(G) \smallsetminus E(G[S])\}$. Clearly, $\dim(W_1) = ||G|| - ||G[S]||$, thus $\dim(W_2) = ||G[S]||$. Furthermore, $W_1$ contains all cuts $\delta(T)$, $T \in \mathcal{F} \smallsetminus \mathcal{F}_S$. Because $|\mathcal{F}| - |\mathcal{F}_S| = ||G|| - ||G[S]||$, these cuts form a basis of $W_1$. Therefore, the remaining cuts $\{\delta(T)\mid T \in \mathcal{F}_S\}$ complete the cuts $\{\delta(T)\mid T \in \mathcal{F} \smallsetminus \mathcal{F}_S\}$ to a basis of $\R^E$ if and only if the `projected' cuts $\delta_{G[S]}(T) = \delta(T) \cap E(G[S])$, $T \in \mathcal{F}_S$ form a basis of $W_2$. Now notice that, since $u$ is the unique neighbor of $v$ in $G[S]$,
$$
\chi^{\delta_{G[S]}(u)} = \chi^{\delta_{G[S]}(\{v,u\})} + \chi^{\delta_{G[S]}(v)}\,.
$$
Because of this, we can modify $\mathcal{F}$ similarly as in the previous case: we add set $\{u,v\}$ and remove singleton $\{u\}$. This again yields a laminar family defining a basis of minimum cuts. The new family $\mathcal{F}$ violates Lemma~\ref{lem:basic_two_cuts} because the edge $vu$ is contained in only one of the cuts it defines, a contradiction. Hence, $G[S]$ does not contain a degree-$1$ node.

\medskip

Thus $G[S]$ is an odd cycle $C$. Since $|E(C)|=|S|$, by Lemma~\ref{lem:contract}, the inequality $\inp{c \contract E(C)}{x} \geqslant k$ is facet-defining for $G \contract E(C)$. Since $G \contract E(C)$ is a proper minor of $G$, there is $\lambda\in\Z_+$, $\lambda\leqslant k$ such that $\frac{\lambda}{k}c(e) \in \Z_+$ for each $e \in E(G) \smallsetminus E(C)$.  

Assume $C = v_1 v_2 \cdots v_t v_1$ and consider the following system in $\R^{E(C)}$
\begin{equation} \label{eqn:oddcycle}
y(v_{i-1}v_{i})+y(v_{i}v_{i+1}) = k-c(\delta(v_i:\overline{S})) \quad \text{for} \quad i=1,\ldots,t
\end{equation}
where indices are taken modulo $t$. Since $\sum_{i=1}^t c(\delta(v_i:\overline{S})) = c(\delta(S)) = k$, one can check that the following is the unique solution of \eqref{eqn:oddcycle}:
$$
y(v_{i}v_{i+1}) = \sum_{j=i+2,i+4,\dots, i+t-1}c(\delta(v_j:\overline{S})) \quad \text{for} \quad i=1,\ldots,t
$$
where, again, indices are taken modulo $t$. Because this solution is unique, we have $c(e)=y(e)$ for $e\in E(C)$. Thus, $\frac{\lambda}{k}c(e) \in \Z_+$ for each $e \in E(C)$, since $\frac{\lambda}{k}c(e) \in \Z_+$ for each $e \in E(G) \smallsetminus E(C)$. This shows that $\frac{\lambda}{k} c$ is an integral vector, contradicting the assumption that $\inp{c}{x}\geqslant k$ is a witness for $G$.
\end{proof}


\begin{rem} \label{rem:level1exist} 
Let $G$ be a minor-minimal non-$k$-graph, $\inp{c}{x} \geqslant k$ be a witness for $G$ and $\mathcal{F}$ be a laminar family such that $\{\delta(S) \mid S \in \mathcal{F}\}$ is a basis of minimum cuts. Moreover, assume that $k \geqslant 2$. Then $\mathcal{F}$ contains at least one level-$1$ set.
\end{rem}
\begin{proof}
By contradiction, assume that all sets in $\mathcal{F}$ are level-$0$. Pick one node $w$ such that $\{w\} \in \mathcal{F}$, remove $\{w\}$ from $\mathcal{F}$ and replace it with $\overline{\{w\}} = V \smallsetminus \{w\}$. The new collection $\mathcal{F}$ contains a level-$1$ set, namely~$\overline{\{w\}}$. Now apply Lemma~\ref{lem:level-1} to the new collection $\mathcal{F}$ and conclude that $G$ has three mutually adjacent nodes $u$, $v$ and $w$ and $c(uv) = c(vw) = c(wu) = \frac{k}{2}$. Therefore, $G = K_3$, $k = 1$ and $c(e) = \frac{1}{2}$ for all edges $e$. This contradicts the assumption $k \geqslant 2$.
\end{proof}


\section{Properties of minor-minimal non-$2$-graphs} \label{sec:non-2-graphs}
 
Here we take $k = 2$ and establish further properties of minor-minimal non-$2$-graphs and their cut dominants. We conclude the section with a proof of Theorem~\ref{thm:2-graphs}, which gives a complete characterization of minor-minimal non-$2$-graphs.
 
\subsection{Half-integrality of witnesses}
  
\begin{lem} \label{lem:half_integral}
Let $G$ be a minor-minimal non-$2$-graph and $\inp{c}{x} \geqslant 2$ be a witness for $G$. Then $c(e) \in \{\frac{1}{2},1\}$ for every $e \in E$.
\end{lem}
\begin{proof}
By Lemma~\ref{lem:basic_cost_one}, it suffices to show that $c(e) \in \frac{1}{2}\Z_+$ for every $e \in E$. As before, let $\mathcal{F}$ be a laminar family such that $\{\delta(S) \mid S \in \mathcal{F}\}$ is a basis of minimum cuts. By Remark~\ref{rem:level1exist}, $\mathcal{F}$ contains a level-$1$ set. By Lemma~\ref{lem:level-1}, this set is of the form $\{u,v\}$ with $\{u\},\,\{v\}\in \mathcal{F}$ and $c(uv) = 1$.

From Lemma~\ref{lem:contract} and from the minimality of $G$, we infer that the inequality $\inp{c \contract uv}{x} \geqslant2$ defines a ridge\footnote{A \emph{ridge} of polyhedron $P$ is a face of dimension $\dim(P) - 2$.} of $\cut{G \contract uv}$. Thus $\cut{G \contract uv}$ has two facet-defining inequalities $\inp{c' }{x} \geqslant 2$, $\inp{c''}{x} \geqslant 2$, such that $c \contract uv$ is a convex combination of $c'$ and $c''$, i.e.:
\begin{equation}\label{eq:half-int-cc}
c\contract uv=\alpha' c'+\alpha'' c''\,,
\end{equation}
for some $\alpha', \alpha'' \in \R_+$ with $\alpha'+\alpha''= 1$. By minimality of $G$, both $c'$ and $c''$ are integral.

Note that we have
\begin{equation}\label{eq:half-int-small-facets}
c'(\delta(v)\smallsetminus uv)+c'(\delta(u)\smallsetminus uv)=c''(\delta(v)\smallsetminus uv)+c''(\delta(u)\smallsetminus uv)=c(\delta(\{u,v\}))=2
\end{equation}
and since, by Lemma \ref{lem:level-1}, $c(uv)=1$,
\begin{equation}\label{eq:half-int-big-facet}
c(\delta(v)\smallsetminus uv)=c(\delta(u)\smallsetminus uv)=1\,.
\end{equation}

Let us show that neither $c'(\delta(u)\smallsetminus uv)$ nor $c''(\delta(u)\smallsetminus uv)$ equals $1$. Suppose that, without loss of generality,  $c'(\delta(u)\smallsetminus uv)=1$ and consider the integral vector $\tilde{c} \in \R^{E}$ defined as
\begin{equation*}
\tilde{c}(e):=\begin{cases}
1& \text{ if } e=uv\\
c'(e)& \text{ otherwise}\,.
\end{cases}
\end{equation*}
It can be checked that $\tilde{c}(\delta(S))=c'(\delta(S))$ and $c'(\delta(S))=2$ for every $S\in\mathcal{F}$ which does not separate $u$ and $v$. Note that in $\mathcal{F}$ there are precisely two sets separating $u$ and $v$, namely $\{u\}$ and $\{v\}$, and note that  $\tilde{c}(\delta(u))=1+c'(\delta(u)\smallsetminus uv)=2$ and consequently by~\eqref{eq:half-int-small-facets} we have $\tilde{c}(\delta(v))=1+c'(\delta(v)\smallsetminus uv)=2$. Thus, $\tilde{c}(\delta(S))=c(\delta(S))$ for every $S \in \mathcal{F}$. Since $\{\delta(S) \mid S \in \mathcal{F}\}$ is a basis, we obtain that $c(e)=\tilde{c}(e)$, and $c$ is an integral vector, a contradiction.

Since $c'$, $c''$ are integral vectors, we have that $c'(\delta(u)\smallsetminus uv)=0$ or $2$ and $c''(\delta(u)\smallsetminus uv) = 0$ or $2$. Then, from~\eqref{eq:half-int-cc}--\eqref{eq:half-int-big-facet} it follows that $\alpha' = \alpha'' = \frac{1}{2}$. Thus, \eqref{eq:half-int-cc} and $c(uv)=1$ imply that $c(e) \in \frac{1}{2} \Z_+$ for every $e \in E$.
\end{proof}

Let $G$ be a minor-minimal non-2-graph and $\inp{c}{x} \geqslant 2$ be a facet-defining inequality of $\cut{G}$ where $c$ is not integral. By Lemma \ref{lem:half_integral}, $\inp{2c}{x} \geqslant 4$ is the minimum integer form of such an inequality. Therefore we have the following:

\begin{rem}
Every minor-minimal non-$2$-graph is a $4$-graph.
\end{rem}

\subsection{Connectivity}

The next remark follows from Remark~\ref{rem:cutdom-cutnode}.

\begin{rem} \label{rem:2-connected_bis}
Every minor-minimal non-2-graph is $2$-connected.
\end{rem}

\begin{lem}[$2$-cutset lemma] \label{lem:2-cutset_bis} Let $G = (V,E)$ be a minor-minimal non-2-graph and $\inp{c}{x} \geqslant 2$ be a witness for $G$. If $G$ has a $2$-cutset $\{u,v\}$ then:
\begin{enumerate}[(i)]
\item $G-\{u,v\}$ has exactly two connected components, one of which contains a single node, say $w$.
\item $u$ and $v$ are not adjacent while $w$ is adjacent to both $u$ and $v$ and $c(uw)=c(vw)=1$.
\item not both $\delta(u)$ and $\delta(v)$ are minimum cuts.
\end{enumerate}
\end{lem}
\begin{proof}
Let $G_1 = (V_1,E_1)$ and $G_2 = (V_2,E_2)$ denote two subgraphs of $G$, each containing at least three nodes, such that
\begin{equation} \label{eq:2-separation}
V_1 \cap V_2 = \{u,v\}, \quad V_1 \cup V_2 = V,\quad
E_1 \cap E_2 = \varnothing \quad \text{and} \quad E_1 \cup E_2 = E\,.
\end{equation}
By Remark \ref{rem:2-connected_bis}, $G_1$ and $G_2$ are both connected.

Let $\mathcal{F}$ be \emph{any} family such that $\{\delta(S) \mid S \in \mathcal{F}\}$ is a basis of minimum cuts. We first show that $\mathcal{F}$ contains a set $S$ such that $\delta(S)$ is a \emph{$u$--$v$ cut}, that is, a cut separating $u$ and $v$. Indeed, otherwise each cut $\delta(S)$, $S\in\mathcal{F}$ is contained in $E_1$ or $E_2$, and hence by Lemma~\ref{lem:contract}, $\inp{c \contract E_i}{x} \geqslant 2$ is a facet-defining inequality for $\cut{G \contract E_i}$, $i = 1, 2$. By minor-minimality of $G$, both vectors $c  \contract E_i$, $i = 1, 2$ are integral, hence $c$ is integral as well, a contradiction.

Fix a $u$--$v$ cut $\delta(S^{*})$ with $S^{*}\in\mathcal{F}$. We show that every $u$--$v$ cut $\delta(S)$, $S \in \mathcal{F}$ may be assumed to satisfy:
\begin{equation}\label{eq:st-cuts-property}
  \delta(S) \cap E_1 = \delta (S^{*}) \cap E_1 \qquad \text{or} \qquad \delta(S) \cap E_2 = \delta(S^{*}) \cap E_2\,.
\end{equation}
Indeed, $\delta(S)$, $S \subseteq V$ is a minimum $u$--$v$ cut in $G$ if and only if $\delta(S) \cap E_i$ is a minimum $u$--$v$ cut in $G_i$ for each $i = 1, 2$. Thus, for every $S\in\mathcal{F}$ the cuts
\begin{equation}\label{eq:two-new-cuts}
\big(\delta(S) \cap E_1\big) \cup \big(\delta(S^{*}) \cap E_2\big) \qquad \text{and} \qquad \big(\delta(S^{*}) \cap E_1\big) \cup \big(\delta(S) \cap E_2\big)
\end{equation}
are minimum $u$--$v$ cuts in $G$ and both these cuts satisfy \eqref{eq:st-cuts-property}. Moreover, the characteristic vectors of the cuts in \eqref{eq:two-new-cuts} sum up to $\chi^{\delta(S)}+\chi^{\delta(S^{*})}$. Hence,  if \eqref{eq:st-cuts-property} does not hold for some $S \in \mathcal{F}$, then $S$ may be removed from $\mathcal{F}$ and replaced by one of the node sets corresponding to the cuts in \eqref{eq:two-new-cuts}, while maintaining linear independence. (We point out that this proof does not use or assume laminarity of $\mathcal{F}$.) Therefore, every $u$--$v$ cut $\delta(S)$, $S \in \mathcal{F}$ may be assumed to satisfy~\eqref{eq:st-cuts-property}.

For $i = 1, 2$, let $G'_i = (V_i,E'_i)$ be the graph obtained by adding a new edge $e'_{3-i}$ with endnodes $u$ and $v$ to $G_i$. By Remark~\ref{rem:2-connected_bis}, each $G'_i$ is a minor of $G$. In fact, because $|V_i| < |V|$, each $G'_i$ is a proper minor of $G$. Notice that $G'_1$ or $G'_2$ may have one pair of parallel edges, but not both. Define $c'_i \in \R^{E'_i}$, $i = 1, 2$ as:

\begin{equation} \label{eq:def_c_prime}
c'_i(e) :=
\begin{cases}
c(\delta(S^*) \cap E_{3-i}) &\text{if } e = e'_{3-i},\\
c(e) &\text{if } e \in E_i\,.
\end{cases}
\end{equation}

By construction, for $i = 1, 2$, the inequality $\inp{c'_i}{x} \geqslant 2$ is valid for $\cut{G'_i}$. We claim that for exactly one index $i \in \{1,2\}$ the inequality $\inp{c'_i}{x} \geqslant 2$ defines a facet of $\cut{G'_i}$.

In order to prove this, consider the square non-singular matrix $M$ whose rows are the (characteristic vectors of the) cuts $\delta(S)$, $S \in \mathcal{F}$. For $i = 1, 2$, let $M_i$ be the submatrix of $M$ induced by the rows corresponding to the cuts whose intersection with $E_{3-i}$ is either empty or equal to $\delta(S^*) \cap E_{3-i}$. Then $M_1$ and $M_2$ are two row-induced submatrices of $M$, thus they have full row-rank. By \eqref{eq:st-cuts-property}, they have exactly one common row, namely, that corresponding to the fixed $u$--$v$ cut $\delta(S^*)$. Hence,
\begin{equation} \label{eq:2-cutset_dim}
\rk(M_1) + \rk(M_2) = \rk(M) + 1 = |E| + 1 = |E_1| + |E_2| + 1\,.
\end{equation}

For $i=1,2$ let $\xi_{3-i}$ denote the 0/1 column vector that has one entry per row of $M_i$, a $1$ at the entries corresponding to $u$--$v$ cuts and a $0$ at the other entries. Notice that the column of $M_i$ associated with any edge $e \in E_{3-i}$ is $\xi_{3-i}$ if $e \in \delta(S^*)$ and  the zero vector otherwise. Let $M'_i$ be the matrix obtained from $M_i$ by removing all columns for edges $e \in E_{3-i}$ and appending a single copy of column $\xi_{3-i}$, indexed by edge $e'_{3-i}$. Notice that the rows of $M'_i$ are characteristic vectors of cuts of $G'_i$. By~\eqref{eq:def_c_prime}, these cuts are minimum with respect to $c'_i$.

Since  $\delta(S^*) \cap E_{3-i}$ is nonempty by Remark~\ref{rem:2-connected_bis}, $\xi_{3-i}$ is a column of $M_i$. Hence $M'_i$ is obtained from $M_i$ by deleting columns that are either $0$ or repeated copies of $\xi_{3-i}$. This shows that $\rk(M_i) = \rk(M'_i)$. Since $E'_i=E_i\cup\{e'_{3-i}\}$, by \eqref{eq:2-cutset_dim} we have that:
\begin{equation} \label{eq:2-cutset_dim M'}
\rk(M'_1) + \rk(M'_2)  = |E'_1| + |E'_2| -1\,.
\end{equation}

Since $\rk(M'_i) \leqslant |E'_i|$ for $i = 1, 2$, we may assume  without loss of generality that $\rk(M'_1) = |E'_1|-1$ and $\rk(M'_2) = |E'_2|$. Then $M'_2$ is a  $|E'_2| \times |E_2'|$ nonsingular matrix and by Lemma~\ref{lem:facet}, $\inp{c'_2}{x} \geqslant 2$ defines a facet of $\cut{G'_2}$. By minimality of $G$, vector $c'_2$ is integral.

Furthermore, $\inp{c'_1}{x} \geqslant 2$ does not define a facet of $\cut{G'_1}$, otherwise $c'_1$ and hence $c$ would be integral. However, since  $\rk(M'_1) = |E'_1|-1$, inequality $\inp{c'_1}{x} \geqslant 2$ defines a ridge of $\cut{G'_1}$.

 Assume $u$ and $v$ are adjacent in $G$. Since $c(uv)$, $c'_2(e'_1)$ and $c'_1(e'_2)=c(\delta(S^*) \cap E_2)$ are all positive and $\inp{c'_2}{x} \geqslant 2$ defines a facet of $\cut{G'_2}$, Lemma~\ref{lem:facet-basic}.(ii) implies that edge $uv$ is in $E_1$. Since the inequality $\inp{c'_1}{x} \geqslant 2$ defines a ridge of $\cut{G'_1}$,  the  inequality $(c(uv) + c'_1(e'_2)) x_{uv} + \sum_{e \in E_1 \smallsetminus \{uv\}} c(e) x(e) \geqslant 2$,  obtained from $\inp{c'_1}{x} \geqslant 2$ by moving all the cost on parallel edges $uv \in E_1$ and $e'_2 \in E'_1 \smallsetminus E_1$ to edge $uv$, defines a facet of $\cut{G'_1}$. By minimality of $G$, the coefficients of this inequality are integral. Since $c'_2$ is an integral vector, $c'_1(e'_2)$ is an integer, which again implies that $c'_1$ is integral, and hence $c$ is integral as well, a contradiction. We conclude that $u$ and $v$ are not adjacent in $G$.

Because $c'_2$ is integral, $c'_2(e_1) = c(\delta(S^{*})\cap E_2) = 1$, since otherwise $c'_2(e_1) \in \{0,2\}$ and one of the values $c(E_i\cap \delta(S^{*}))$, $i=1,2$ is equal to $0$, showing that $u$ or $v$ is a cutnode and contradicting Remark~\ref{rem:2-connected_bis}.

Now, we show $|G_2|=3$. Assume not, then the graph $G''$ with node set $V_1 \uplus \{w\}$ and edge set $E_1 \cup \{uw,vw\}$ is a proper minor of $G$. The inequality $\inp{c''}{x}\geqslant 2$ is valid for $\cut{G''}$, where
\begin{equation*}
c''(e):=\begin{cases} c(e) &\text{if } e\in E_1\\
1 &\text{if } e \in \{uw,vw\}\,.
\end{cases}
\end{equation*}
Let $M''$ denote the matrix obtained from $M'_1$ by re-indexing with $uw$ the column indexed by $e'_2$, appending a new column indexed by $wv$ that is everywhere $0$ and finally adding two new rows for the cuts $\delta_{G''}(S^*) = (\delta(S^*) \cap E_1) \cup \{wv\}$ (assuming, without loss of generality, that $u \in S^*$ and $w \notin S^*$) and $\delta_{G''}(w) = \{uw,wv\}$. We leave it to the reader to check that $\rk(M'') = \rk(M_1) + 2 = |E_1| + 2$. The rows of $M''$ form a basis of cuts of $G''$ that are minimum with respect to $c''$. Hence, $\inp{c''}{x} \geqslant 2$ defines a facet of $\cut{G''}$. Due to minimality of $G$, vector $c''$ is integral, leading to integrality of $c$, and hence to a contradiction.

Thus, $|G_2| = 3$, that is, $V_2$ consists of $u$, $v$ and one other node $w$, and in $G$ the node $w$ is incident only to $u$ and $v$. Combining $c(uw)+c(vw)=c(\delta(w))\geqslant 2$ with $c(uw)\leqslant 1$, $c(vw)\leqslant 1$ (see Lemma~\ref{lem:basic_cost_one}), we get $c(uw) = c(vw) = 1$ as desired. 

Let us show (iii). For this assume that both $\{u\}$ and $\{v\}$ define minimum cuts in $G$. In this case, $\delta_{G'_1}(u)$, $\delta_{G'_1}(v)$ and $\delta_{G'_1}(\{u,v\})$ are minimum cuts in $G'_1$, since 
$$
c'_1(\delta_{G'_1}(\{u,v\}))=c(\delta_G(u))+c(\delta_G(v))-c(uw)-c(vw)=2\,.
$$
Thus, $\inp{c'_1}{x}\geqslant 2$ defines a facet of $\cut{G'_1}$. Indeed, the minimum cuts corresponding to the rows of $M'_1$ together with the cuts $\delta_{G'_1}(u)$, $\delta_{G'_1}(v)$ and $\delta_{G'_1}(\{u,v\})$ span $\R^{E'_1}$, because the rows of $M'_1$ restricted to $E_1$ span $\R^{E_1}$, and 
$$
\chi^{e'_2}=\frac{1}{2}\big(\chi^{\delta_{G'_1}(u)}+\chi^{\delta_{G'_1}(v)}-\chi^{\delta_{G'_1}(\{u,v\})}\big)\,.
$$ 
Due to minor-minimality of $G$, $c'_1$ is integral. Since both $c'_1$ and $c'_2$ are integral we conclude  that the vector $c$ is integral, a contradiction.

It remains to prove that $G - \{u,v\}$ has exactly two components. If $G - \{u,v\}$ has more than two components, then by applying the reasoning above to all possible pairs of subgraphs $G_1$, $G_2$ satisfying \eqref{eq:2-separation}, we conclude that $G - \{u,v\}$ has exactly three components, each of which containing a single node connected to both $u$ and $v$ by an edge of cost~$1$. Thus $c$ is integral (and $G = K_{3,2}$), a contradiction.
\end{proof}


\medskip
\subsection{Level-$2$ sets}

We need the following basic but useful result proved in~\cite{FN92}.

\begin{lem} \label{lem:star_cycle}
Let $G = (V,E)$ be a connected graph with three nodes $v_1, v_2, v_3$ such that $G-v_i$ is connected for $i = 1, 2, 3$. Then $G$ contains at least one of the two following minors: a claw $K_{1,3}$ with nodes $u^*$, $v_1^*$, $v_2^*$, $v_3^*$ and center $u^*$ or a cycle on three nodes $v_1^*$, $v_2^*$, $v_3^*$. Moreover, in the operation which transforms $G$ into one of these two minors $v_i$ is mapped onto $v_i^*$ for $i = 1,2,3$.
\end{lem}

Now we show that family $\mathcal{F}$ can be assumed to contained no level-$2$ set.

\begin{lem} \label{lem:level-2}
Let $G = (V,E)$ be a minor-minimal non-2-graph, let $\inp{c}{x} \geqslant 2$ be a witness for $G$, and let $\mathcal{F}$ be a laminar family such that $\{\delta(T) \mid T \in \mathcal{F}\}$ is a basis of minimum cuts. Assume that $\mathcal{F}$ contains a level-$2$ set $S$. Then, one of the following holds:

\begin{enumerate}[(i)]
\item $S$ is maximal in $\mathcal{F}$ and replacing $S$ by $\overline{S}$ in $\mathcal{F}$ gives a family without level-$2$ set, or
\item $G$ has the prism or pyramid as minor. 
\end{enumerate}
\end{lem}
\begin{proof}
We assume that (i) does not hold and prove that (ii) holds. By assumption we have both $|S| \geqslant 3$ and $|\overline{S}| \geqslant 3$. We claim that $\delta(S)$ contains a matching of size~$3$. Indeed, otherwise by K\"onig's Theorem there exists a set $\{u,v\}$ such that every edge in $\delta(S)$ has an endpoint in $\{u,v\}$. By symmetry, it suffices to consider the two following cases.\medskip

\noindent \emph{Case 1}: $u \in S$ and $v \in \overline{S}$. Then $G - \{u,v\}$ has either more than two components or two components with at least~$2$ nodes each, contradicting Lemma~\ref{lem:2-cutset_bis}.\medskip

\noindent \emph{Case 2}: $u, v \in S$. Again, $\{u,v\}$ is a $2$-cutset of $G$. Lemma~\ref{lem:2-cutset_bis} implies that $S = \{u,v,w\}$ for some degree-$2$ node $w$ that is adjacent to both $u$ and $v$. Because $\chi^{\delta(S)} + \chi^{\delta(w)} = \chi^{\delta(u)} + \chi^{\delta(v)}$ and $S$ is in $\mathcal{F}$, only two of the sets $\{u\}$, $\{w\}$, $\{v\}$ can belong to $\mathcal{F}$. Because $S$ is level-$2$, it contains a level-$1$ set $T$. By Lemma~\ref{lem:level-1}, $T$ is either $\{u,w\}$ or $\{w,v\}$. Without loss of generality, we have $T = \{u,w\}$ and hence $\{u,w\}, \{u\}, \{w\} \in \mathcal{F}$, $\{v\} \notin \mathcal{F}$. Due to $\chi^{\delta(S)} + \chi^{\delta(w)} = \chi^{\delta(\{v,w\})} + \chi^{\delta(\{u,w\})}$, the laminar collection $\mathcal{F}'$ obtained from $\mathcal{F}$ by replacing $\{u,w\}$ by $\{v,w\}$ also yields a basis of minimum cuts. Since only one of the cuts in $\mathcal{F}'$ contains the edge $wv$, $\mathcal{F}'$ does not satisfy Lemma~\ref{lem:basic_two_cuts}, a contradiction.
\medskip

We conclude that $\delta(S)$ contains a matching of size~$3$. Let this matching be formed by the edges $u_iv_i$ with $u_i \in S$, $v_i \in \overline{S}$, for $i = 1, 2, 3$. Next, we claim that $G[S] - u_i$ is connected for every $i = 1, 2, 3$, and similarly $G[\overline{S}] - v_i$ is connected for all $i = 1, 2, 3$. Assume the opposite and let  $G[S] - u_1$ be not connected. Thus $G[S] - u_1$ is the union of two disjoint graphs $G_1=(V_1,E_1)$ and $G_2=(V_2,E_2)$ with at least one node each. Then,
\begin{equation*}
c(\delta(S)) = c(\delta(V_1:\overline{S})) + c(\delta(V_2:\overline{S})) + c(\delta(u_1:\overline{S})) = 2\,.
\end{equation*}
By Remark~\ref{rem:2-connected_bis}, $c(\delta(V_2:\overline{S}))>0$ and $c(\delta(V_1:\overline{S}))>0$. Moreover, $c(\delta(u_1:\overline{S})) \geqslant c(u_1v_1) > 0$. Hence, due to the half-integrality of $c$ (see Lemma~\ref{lem:half_integral}) there is $i \in \{1,2\}$ such that $\delta(V_i:\overline{S})$ consists of a single cost-$1/2$ edge. Without loss of generality, let $\delta(V_1:\overline{S})$ consist of a single cost-$1/2$ edge $wv$, $w\in V_1$, $v\in \overline{S}$. Then $\{u_1,v\}$ is a $2$-cutset in the graph $G$. Since $c(wv) = \frac{1}{2}$, it contradicts Lemma~\ref{lem:2-cutset_bis}. We conclude that $G[S] - u_i$ and $G[\overline{S}] - v_i$ are connected for $i = 1, 2, 3$.

By Lemma~\ref{lem:star_cycle}, it follows that $G[S]$ contains one of the following subgraphs:
\begin{itemize}
\item a cycle through $u_1$, $u_2$ and $u_3$, or
\item a subdivided claw with leaves $u_1$, $u_2$ and $u_3$,
\end{itemize}
and similarly for $G[\overline{S}]$. Together with the matching $\{u_iv_i \mid i = 1,2,3\}$ these subgraphs yield either a prism or a pyramid minor in $G$ and hence (ii) holds, except perhaps in the case where $G[S]$ and $G[\overline{S}]$ both contain a subdivided claw. 

In this case, we let $u_0$ and $v_0$ denote the centers of these subdivided claws and use Lemma~\ref{lem:2-cutset_bis}. By the lemma, $G - \{u_0,v_0\}$ has only one component with at least two nodes. Thus $u_i$, $v_i$, $i=1,2,3$ should be in the same component of $G - \{u_0,v_0\}$. Therefore, $G$ contains more edges besides those of the subdivided claw and matching. If $G[S] - u_0$ contains a $u_i$--$u_j$ path for $i \neq j$ then $G$ contains a pyramid minor, and (ii) holds, and similarly for $G[\overline{S}] - v_0$. Otherwise, we conclude that $\delta(S)$ contains at least five edges: the three edges of the matching $\{u_iv_i \mid i = 1,2,3\}$ and two more. However, by Lemma~\ref{lem:half_integral}, $\delta(S)$ contains at most four edges, a contradiction.
\end{proof}

\medskip
\subsection{Final Analysis}

\begin{proof}[Proof of Theorem~\ref{thm:2-graphs}]
Let $G = (V,E)$ be a minor-minimal non-$2$-graph, $\inp{c}{x} \geqslant 2$ be a witness for $G$ and $\mathcal{F}$ be a laminar family such that the cuts $\delta(S)$, $S \in \mathcal{F}$ form a basis of minimum cuts. We show that $G$ contains the prism or pyramid as minor.

We use the following notation: $E_\gamma := \{e \in E \mid c(e)=\gamma\}$ for $\gamma \in \{1,1/2\}$. By Lemma~\ref{lem:half_integral}, $c(e) \in \{1,\frac{1}{2}\}$ for all $e \in E$. So $E = E_{1/2} \cup E_1$.

Due to Lemma~\ref{lem:level-2}, we may assume that $\mathcal{F}$ has no level-$2$ set. Hence, we can partition the nodes of $G$ into three subsets $V_{-1}$, $V_0$ and $V_1$:
\begin{itemize}
\item $V_{-1}$ is the set of nodes that are not covered by sets in $\mathcal{F}$
\item $V_{0}$ is the set of nodes that are covered by a level-$0$ set in $\mathcal{F}$ and by no level-$1$ set in $\mathcal{F}$
\item $V_{1}$ is the set of nodes that are covered by a level-$1$ set in $\mathcal{F}$.
\end{itemize}

By Remark~\ref{rem:level1exist}, the family $\mathcal{F}$ contains a level-$1$ set, say $S$. By Lemma~\ref{lem:level-1}, $S = \{u,v\}$ where $uv \in E_1$ and $\{u\}, \{v\} \in \mathcal{F}$. Lemma~\ref{lem:2-cutset_bis} shows that $G$ does not contain a path of length greater or equal to three. Hence, $\delta(S)$ contains at least one cost-$1/2$ edge. Thus, we assume without loss of generality that one of the following is true: $\delta(S)\cap E_{1/2} = \delta(v:\overline{S})$ or $\delta(S)\cap E_{1/2} =\delta(S)$.

If $\delta(S)\cap E_{1/2} =\delta(v:\overline{S})$, node $u$ is adjacent to exactly two nodes: $v$ and some node $w$.  By Lemma~\ref{lem:2-cutset_bis}, $\delta(w)$ is not a minimum cut. Thus $w$ lies in $V_{-1}$. 

Note also that the degree of every $v \in V_{0}$ in the graph $G$ is at least $3$. Indeed, assume that is not true, then the node $v$ is adjacent to exactly two nodes $u$ and $w$ and the edges $uv$, $wv$ are cost-$1$ edges, since  $c(\delta(v))\geqslant2$ and $c(e) \leqslant 1$ for all edges $e$ (see Lemma~\ref{lem:basic_cost_one}). Moreover, by Lemma~\ref{lem:basic_two_cuts} neither $u$ nor $w$ lies in $V_{-1}$. 
On the other side, by Lemma~\ref{lem:2-cutset_bis} not both $\delta(u)$ and $\delta(w)$ are minimum cuts, hence at least one of $u$ and $w$ lies in $V_{-1}$, a contradiction.

\begin{claim*}\label{lem:charging}
Both $V_0$ and $E_{1/2}\cap \delta(V_{-1}:V_{1})$ are empty.
\end{claim*}
\begin{proof}
Let us prove this by a discharging argument. Give to each of the edges in $G$ a unit charge, so the total charge is $|E|$.

Now let the edges of $G$ distribute their charge to the sets of $\mathcal{F}$ in the following manner: the charge of an edge $e:=uv$ is equally distributed between the maximal element of $\mathcal{F}$ containing $u$ and the maximal element of $\mathcal{F}$ containing $v$. If there is no element of $\mathcal{F}$ that contains $u$, then the whole charge of the edge is transmitted to the maximal element of $\mathcal{F}$ containing $v$; and similarly if there is no element of $\mathcal{F}$ that contains $v$. Thus, if the maximal elements of $\mathcal{F}$ containing respectively $u$ and $v$ are equal, then the whole value of the edge is transmitted to this maximal element. Note, that the total charge of the sets in $\mathcal{F}$ after this procedure is $|E|$, since by Lemma~\ref{lem:basic_two_cuts} there is no edge $e:=uv$ in $G$ such that both $u$ and $v$ lie in $V_{-1}$.

Every node $v\in V_0$,  is contained in only one set in $\mathcal{F}$, namely in the set $\{v\}$. Moreover, from the discussion before we know that $v$ is incident to at least three edges, and thus the total charge of the set $\{v\}$ is at least $3/2$ for every $v\in V_0$.

By Lemma~\ref{lem:level-1}, for every node $v\in V_1$ there is a unique node $u \in V_1$, such that $S:=\{v,u\}$ is inclusionwise maximal in $\mathcal{F}$. Recall, that there are two cases for the set $S$.\medskip

\noindent \emph{Case 1}: $u$ is adjacent to $v$ by a cost-$1$ edge and to one node in $V_{-1}$ by another cost-$1$ edge; besides $u$, node $v$ is adjacent to two other nodes by cost-$1/2$ edges. Thus, the charge of $S$ is equal to $3+\frac{1}{2} |E_{1/2}\cap \delta(V_{-1}:S)|$.\medskip

\noindent \emph{Case 2}:  $uv$ is a cost-$1$ edge and both $u$, $v$ are incident to two cost-$1/2$ edges. In this case, the charge of the set $S$ also equals $3+\frac{1}{2} |E_{1/2}\cap \delta(V_{-1}:S)|$.\medskip

Hence, the total charge of the family $\mathcal{F}$ is at least
\begin{align*}
\frac{3}{2}|V_0|+\sum_{\substack{S\in \mathcal{F}\\ \level(S)=1}}\big( 3+\frac{1}{2} |E_{1/2}\cap \delta(V_{-1}:S)|\big)=\frac{3}{2}|V_0|+\frac{3}{2}|V_1|+\frac{1}{2}|E_{1/2}\cap \delta(V_{-1}:V_1)|\,.
\end{align*}
But the total charge should be equal to $|E|$, and thus to the cardinality of $\mathcal{F}$, which in its turn is equal to $|V_0|+ \frac{3}{2}|V_1|$. Combining this with the estimation of the charge from below we get the following inequality
\begin{equation*}
|V_0|+\frac{3}{2} |V_1|\geqslant\frac{3}{2}|V_0|+\frac{3}{2}|V_1|+\frac{1}{2}|E_{1/2}\cap \delta(V_{-1}:V_1)|\,,
\end{equation*}
which proves that the sets $V_0$ and $E_{1/2}\cap \delta(V_{-1}:V_1)$ are empty.
\end{proof}

Now consider the graph $G \delete E_1$. By Lemma~\ref{lem:basic_two_cuts}, no edge of $G$ has both ends in $V_{-1}$. By the claim, the set $V_0$ is empty and all the edges in $\delta(V_1:V_{-1})$ are cost-$1$ edges. Thus, in the graph $G \delete E_1$ each node of $V_{-1}$ has degree~$0$ and each node of $V_1$ has degree~$0$ or $2$.

Hence, the edges of $G \delete E_1$ form a collection of node-disjoint cycles. If there are at least two distinct cycles $C'$, $C''$ in the collection then, by Lemma~\ref{lem:2-cutset_bis}, $G$ contains three node-disjoint paths connecting  $V(C')$ and $V(C'')$, showing that $G$ has the prism as minor.

Thus, we may assume that the edges of $G\delete E_1$ form one cycle $C$. 
The cycle $C$ is an odd cycle, since the cuts $\delta(S)$, $S \in \mathcal{F}$ are linearly independent and for each level-$1$ set $S = \{v,w\}$ in $\mathcal{F}$, $\chi^{\delta(S)}, \chi^{\delta(v)}, \chi^{\delta(w)}$ have the same linear span as $\chi^{\{vw\}}, \chi^{\delta(v) \smallsetminus \{vw\}}, \chi^{\delta(w) \smallsetminus \{vw\}}$. Since $|V_{1}|$ is even, there is a node $v_1 \in V_{1}$ with $v_1 \not\in V(C)$. In $G$, node $v_1$ is adjacent to exactly two nodes by cost-$1$ edges, one of which lies in $V_{-1}$, say node $u$. Because $V_{-1}$ is a stable set in $G$ and all the edges of $\delta(V_1:V_{-1})$ are cost-$1$ edges, and due to Lemma~\ref{lem:2-cutset_bis}, $u$ is adjacent to $v_1$ and to at least two other nodes in $V_1$, say $v_2$ and $v_3$. Because all the nodes of $C$ are incident to two cost-$1/2$ edges and one cost-$1$ edge, and $v_i$, $i = 1, 2, 3$ is incident to two cost-$1$ edges, none of the $v_i$'s is a node of $C$. Letting $w_i$, $i = 1, 2, 3$ denote the unique neighbor of $v_i$ within $V_1$, we see that each of the $w_i$'s is a node of $C$. Then $C$ together with the three paths $u, v_i, w_i$, $i = 1, 2, 3$ yield a pyramid minor in $G$. The theorem follows.
\end{proof}

\section{Equivalence} \label{sec:equivalence}

Here we discuss the links between our forbidden minor characterization of $2$-graphs (Theorem~\ref{thm:2-graphs}) and Fonlupt and Naddef's forbidden minor characterization of TSP-perfect graphs (Theorem~\ref{thm:Fon-Nadd}). We prove that each theorem implies the other one in the sense that, with some overhead, \emph{any} proof of one of the theorems yields a proof of the other one. In the reductions we give, the overhead is most substantial when reducing Theorem~\ref{thm:2-graphs} to Theorem~\ref{thm:Fon-Nadd} because we make use of our $2$-cutset lemma (Lemma~\ref{lem:2-cutset_bis}). 

\begin{prop}
Theorem~\ref{thm:2-graphs} implies Theorem~\ref{thm:Fon-Nadd}.
\end{prop}
\begin{proof}
Let $G = (V,E)$ be a $2$-graph, and let $\inp{c}{x} \geqslant 2$ define a facet of $\cut{G}$ with full support. Since $\lambda^c(G) = 2$, graph $G$ is connected. By Lemma~\ref{lem:facet-basic}, $G$ is simple and $c(e) \in \{1,2\}$ for all edges $e$. 

Assume that $G$ has two nodes $u$ and $v$ such that there are three internally disjoint paths $P_i$, $i = 1, 2, 3$ between $u$ and $v$ (that is, $P_i$ and $P_j$ have only nodes $u$ and $v$ in common for $i \neq j$). If $P_i$ is a single edge for some $i = 1, 2, 3$, then $P_i = uv$ and there is no minimum cut containing $uv$, contradicting the fact that $\inp{c}{x} \geqslant 2$ is facet-defining. Otherwise, if $P_i = uwv$ for some $i = 1, 2, 3$ and some node $w$, then there is at most one minimum cut containing $uw$ or $wv$, namely $\delta(\{w\})$. Hence all minimum cuts satisfy $x(uw) = x(wv)$, contradicting again the fact that $\inp{c}{x} \geqslant 2$ is facet-defining. Otherwise, each $P_i$, $i = 1, 2, 3$ has at least three edges and $G$ contains a \Mone{} minor.

Now assume that $G$ does not have two nodes with three internally disjoint paths between them. Then no two cycles of $G$ share an edge, that is, every $2$-connected component of $G$ is a cycle. In this case $c(e) = 1$ for all edges $e$ contained in a cycle and $c(e) = 2$ for all edges $e$ that are bridges, so that $c(\delta(v))$ is even for all nodes $v$. 
\end{proof}

\begin{prop}
Theorem~\ref{thm:Fon-Nadd} implies Theorem~\ref{thm:2-graphs}.
\end{prop}
\begin{proof}
By contradiction, suppose there is a minor-minimal non-$2$-graph $G$ that does not have a prism or pyramid minor. By Theorem~\ref{thm:Fon-Nadd}, $G$ has a \Mone{} minor, hence two nodes $u$ and $v$ and three internally disjoint $u$--$v$ paths $P_1, P_2, P_3$ each with at least three edges. Let $u_i$ and $v_i$ denote the neighbours of $u$ and $v$ respectively on the path $P_i$, for $i = 1, 2, 3$.

By Lemma~\ref{lem:2-cutset_bis}, the paths $P_i - \{u,v\}$, $i = 1, 2, 3$ belong to the same connected component of $G - \{u,v\}$. Thus, we may assume that $G$ has both a path $Q_{12}$ between  $P_1$ and $P_2$ and a path $Q_{23}$ between $P_2$ to $P_3$, where $Q_{12}$ and $Q_{23}$ avoid the nodes  $u$, $v$. 

For pairwise distinct $i, j, k=1, 2, 3$ there is no path between $P_i$ and $P_j$, which avoids $V(P_k)\cup \{u_i, u_j\}$, because the graph $G$ does not have a pyramid minor. Analogously, for pairwise distinct $i, j, k=1, 2, 3$ there is no path between $P_i$ and $P_j$, which avoids $V(P_k)\cup \{v_i, v_j\}$.  

Moreover, if for pairwise distinct $i, j, k=1, 2, 3$ there are a $v_i$--$u_j$ path and a $v_j$--$u_k$ path, which are internally disjoint and both avoid the nodes $\{u, v, u_i, v_k\}$, then the graph $G$ has a prism minor. The same statements holds when the roles of $u$ and $v$ are exchanged.

Thus, without loss of generality we may assume that the endpoints of $Q_{12}$ are $v_1$ and $u_2$ and that the endpoints of $Q_{23}$ are $u_2$ and $v_3$. By the same reasoning, $Q_{12}$ and $Q_{23}$ are both internally disjoint from $P_1 \cup P_2 \cup P_3$ and from each other. Continuing in the same vein, we see that $\{u,v_3\}$, $\{u,v_1\}$ and $\{v, u_2\}$ are $2$-cutsets of $G$. Hence, each of the paths $P_i$, $i = 1, 2, 3$ has length exactly $3$, see Figure~\ref{fig:3-paths}.

Now consider a witness $\inp{c}{x} \geqslant 2$ for $G$. By Lemmas~\ref{lem:half_integral} and \ref{lem:2-cutset_bis}, the weights of the edges incident to any one of $u_1$, $v_2$, $u_3$ are all $1$ while the weights of all other edges are at least $1/2$. But then the edge $uu_2$ is in no minimum cut, contradicting Lemma~\ref{lem:facet-basic}.
\end{proof}

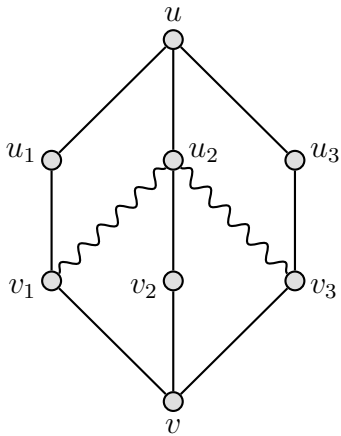
\begin{figure}[ht]
\centering
\begin{tikzpicture}[inner sep = 2.5pt,thick, scale=1.6]
\tikzstyle{vtx}=[circle,draw,thick,fill=gray!25]
\node[vtx] (s) at (0,0) {};
\node[vtx] (t) at (0,3) {};
\node[vtx] (v1) at (-1,1) {};
\node[vtx] (v2) at (0,1) {};
\node[vtx] (v3) at (1,1) {};
\node[vtx] (w1) at (-1,2) {};
\node[vtx] (w2) at (0,2) {};
\node[vtx] (w3) at (1,2) {};
\draw (s) -- (v1) -- (w1) -- (t); 
\draw (s) -- (v2) -- (w2) -- (t); 
\draw (s) -- (v3) -- (w3) -- (t); 
\draw[snake it] (v1) -- (w2); 
\draw[snake it] (w2) -- (v3); 
\node[below] at (s.south) {$v$};
\node[above] at (t.north) {$u$};
\node[left] at (v1.south west) {$v_1$};
\node[left] at (w1.north west) {$u_1$};
\node[left] at (v2.south west) {$v_2$};
\node[right] at (w2.north east) {$u_2$};
\node[right] at (v3.south east) {$v_3$};
\node[right] at (w3.north east) {$u_3$};
\end{tikzpicture}
\caption{Structure of a minor-minimal non-$2$-graph with a \Mone{} minor.}
\label{fig:3-paths}
\end{figure}

Phrased in traveling salesman language, the most important difference between our main result and Fonlupt and Naddef's main result seems to be as follows. We show that forbidding the prism and pyramid as minors already guarantees the integrality of $\subtour{G}$, and that forbidding also a \Mone{} minor implies $\subtour{G} = \GTSP{G}$. Fonlupt and Naddef~\cite{FN92} work directly with three forbidden minors.

\section{Concluding Remarks} \label{sec:concluding_remarks}

There exists a family $(H_n)_{n \geqslant 6,\;n\;\mathrm{even}}$ of planar graphs such that $|H_n| = n$ and $k^*(H) = 2^{\Omega(n)}$, see Figure~\ref{fig:caterpillar}. This example is adapted from a known integrality gap example, see~\cite{CKKK08}. By the Excluded Grid Theorem \cite{RS86}, this proves that for every fixed $k$, all $k$-graphs have constant tree-width. In other words, there exists a function $f$ such that 
\begin{equation}\label{eq:tree_width}
	\mathrm{tw}(G) \leqslant f(k^*(G))
\end{equation}
for all graphs $G$. 

This has algorithmic consequences. For instance, solving TSP instances on graphs $G$ such that the vertices of $\subtour{G}$ have bounded fractionalities can be done in polynomial time. This is due to the fact that such graphs have $k^*(G) = O(1)$ and thus $\mathrm{tw}(G) = O(1)$ and the TSP can be solved in polynomial time via dynamic programming.

However, the converse of \eqref{eq:tree_width} does not hold: since the graphs $H_n$ have constant tree-width (even constant path-width), bounding $\mathrm{tw}(G)$ does not guarantee that $k^*(G)$ is bounded. Thus, there is no function $g$ such that  $k^*(G)\leqslant g(\mathrm{tw}(G))$.

\begin{figure}[ht]
\begin{center}
 	  \begin{tikzpicture}[inner sep = 2.5pt, xscale = 2, yscale = 0.8]
      \tikzstyle{vtx}=[circle,draw,thick,fill=gray!25]
      \foreach \i in {2,...,6} {
        \draw[very thick,magenta] (\i,0) -- (\i+1,0);
        \draw[very thick,magenta] (\i,0) -- (\i,3);
      }
      \foreach \i in {2,...,5} {
        \draw[thick] (\i,3) edge[bend right = 12] (\i+2,0);
        \draw[thick] (\i,3) -- (\i+1,3);
      }
      \draw[very thick,magenta] (1,0) -- (2,0);
      \draw[thick] (1,0) -- (2,3);
      \draw[thick] (6,3) -- (7,0);
      \draw[thick] (1,0) edge[bend right = 60] (3,0);
      \foreach \i in {2,...,6} {
        \node[vtx] at (\i,1.5) {};
        \node[vtx] at (\i,3) {};
        \node[vtx] at (\i+.5,0) {};
        \node[vtx] at (\i,0) {};
      }
      \node[vtx] at (1,0) {};
      \node[vtx] at (1.5,0) {};
      \node[vtx] at (7,0) {};
      \node at (1.45,2.1) {\small $\frac{1}{2}$};
      \node at (2,-1) {\small $\frac{1}{2}$};
      \node at (2.5,3.45) {\small $\frac{1}{4}$};
      \node at (3.5,3.45) {\small $\frac{3}{8}$};
      \node at (4.5,3.45) {\small $\frac{5}{16}$};
      \node at (5.5,3.45) {\small $\frac{11}{32}$};
      \node at (2.5,1.5) {\small $\frac{1}{4}$};
      \node at (3.5,1.5) {\small $\frac{3}{8}$};
      \node at (4.5,1.5) {\small $\frac{5}{16}$};
      \node at (5.5,1.5) {\small $\frac{11}{32}$};
      \node at (6.65,2.25) {\small $\frac{21}{32}$};
    \end{tikzpicture} 
\end{center}
\caption{Graph $H_{12}$ together with a facet-defining inequality $\inp{c}{x} \geqslant 2$ for $\cut{H_{12}}$. We let $c(e) = 1$ for every magenta edge $e$. The costs $c(e)$ of the other edges $e$ are as indicated on the figure.}
\label{fig:caterpillar}
\end{figure}
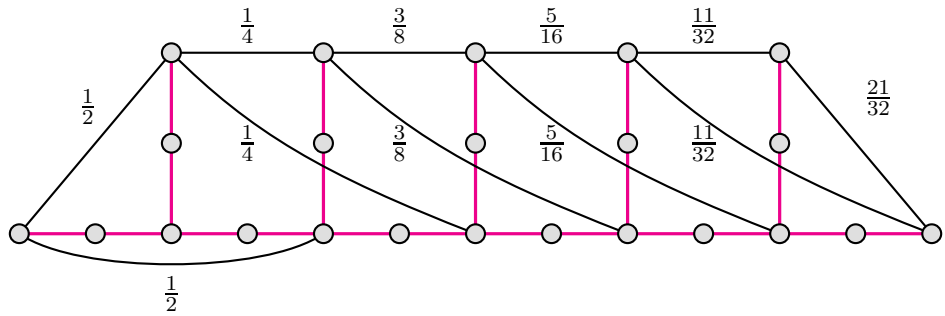
  	
%

\bibliographystyle{plain}

\bibliography{cut-dom}

\end{document}